\documentclass[12pt]{amsart}
 \usepackage[dvips]{epsfig}
 \usepackage{enumerate}
 \usepackage{amsgen, amstext,amsbsy,amsopn, amsthm, amsfonts,amssymb,amscd,amsmat
 h,euscript,enumerate,url,verbatim,calc,xypic, mathtools}

 \usepackage{latexsym}
 \usepackage{graphics}
 \usepackage{color}

\newcommand{\proset}{\,\mathrel{\lower 4pt\hbox{$\scriptscriptstyle/$}
\mkern -14mu\subseteq }\,} 
 
 \newtheorem{theorem}{Theorem}[section]
  
 \newtheorem{lemma}[theorem]{Lemma}

\newtheorem{remark}[theorem]{Remark}
 
 \newtheorem{definition}[theorem]{Definition}

\numberwithin{equation}{section}

\def\ker{\operatorname{ker}}

\usepackage{amsmath}
\usepackage{tikz-cd}
 \makeatother 
 
\begin{document}
\date{\today}

\title{On the generators of Nil $K$-groups}
 \author{Sourayan Banerjee and Vivek Sadhu} 
 
 \address{Department of Mathematics, Indian Institute of Science Education and Research Bhopal, Bhopal Bypass Road, Bhauri, Bhopal-462066, Madhya Pradesh, India}
 \email{sourayan16@iiserb.ac.in, vsadhu@iiserb.ac.in, viveksadhu@gmail.com}

 \begin{abstract}
  In this article, we study higher Nil $K$-groups via binary complexes. More particularly, we exhibit an explicit form of generators of higher Nil $K$-groups in terms of binary complexes.
  \end{abstract}

\maketitle
 
\section{Introduction}

Given a commutative ring $R,$ let ${\bf P}(R)$  denote the category of finitely generated projective $R$-modules. Let ${\bf Nil}(R)$ be a category consisting of all pairs $(P, \nu),$ where $P$ is a finitely generated $R$-module and $\nu$ is a nilpotent endomorphism of $P.$ Let ${\rm Nil}_{0}(R)$ denote the kernel of the forgetful map $K_{0}({\bf Nil}(R)) \to K_{0}({\bf P}(R))=: K_{0}(R).$ The group ${\rm Nil}_{0}(R)$ is generated by elements of the form $[(R^{n}, \nu)]- [(R^{n}, 0)]$ for some $n$ and some nilpotent endomorphism $\nu.$ Using these generators, we show that ${\rm Nil}_{0}(R)=0$ provided every finitely generated torsion free $R$-module is projective (see Theorem \ref{nil zero vanish}). The hypothesis on $R$ in the above-mentioned result holds for many well-known classes of rings, such as PID, valuation rings, Dedekind domains, and so on. In fact, in case of integral domains, the hypothesis on $R$ is equivalent to $R$ being a Pr\"{u}fer domain (i.e., a ring which is locally a valuation ring). One of the goals of this article is to determine the generators of higher Nil $K$-groups.

Given an exact category $\mathcal{N},$ let $i\mathcal{N}$ denote the subcategory of $\mathcal{N}$ whose arrows are isomorphisms. Consider $i\mathcal{N}$ as a category of weak equivalences. Then the $K$-theory spectrum of $\mathcal{N}$ defined as $K\mathcal{N}:= Ki\mathcal{N}.$  
The $n$-th $K$-group of $\mathcal{N}$ is defined as $K_{n}\mathcal{N}:= \pi_{n}(Ki\mathcal{N})= K_{n}i\mathcal{N},$ where $n\geq 0$ (see Appendix A of \cite{Gray}). Let ${\rm Nil}(R)$ be the homotopy fibre of the forgetful functor $K{\bf Nil}(R) \to  K{\bf P}(R)=: K(R).$ The $n$-th Nil group ${\rm Nil}_{n}(R)$ is $\pi_{n}{\rm Nil}(R).$ Since the forgetful functor splits,  $K{\bf Nil}(R) \simeq {\rm Nil}(R) \times K(R).$  This implies that $K_{n}{\bf Nil}(R)\cong {\rm Nil}_{n}(R)\bigoplus K_{n}(R)$ for every ring $R.$ 
We also have a natural decomposition $K_{n}(R[t])\cong K_{n}(R)\oplus NK_{n}(R),$ where $NK_{n}(R)= \ker [K_{n}(R[t]) \stackrel{t\mapsto 0}\to K_{n}(R)].$ There is an isomorphism  ${\rm Nil}_{n}(R)\cong NK_{n+1}(R)$ for all $n$ and $R $ (see Theorem V.8.1 of \cite{wei 1}). The Nil $K$-groups measure the failure of algebraic $K$-theory to be homotopy invariant. The group ${\rm Nil}_{n}(R)$ is not finitely generated unless it is trivial (for instance, see Proposition IV.6.7.4 of \cite{wei 1}). It is natural to wonder:
\begin{center}
      {\it How do the generators of the group ${\rm Nil}_{n}(R)$ for $n>0$ look like}?                                                                                                                                                                                                                                                                                                                                                                                                                                                                                                                                                                                                                                                                                                                                                                       
                                                                                                                                                                                                                                                        \end{center}

 In this article, we are able to figure out generators of ${\rm Nil}_{n}(R)$ for $n>0$ using Grayson's technique (see \cite{Gray}) of binary complexes. We refer to section \ref{Nil1} (Theorems \ref{gen for Nil1} and \ref{gen for Nil_n}) for the precise results.  We hope the explicit form of generators of Nil $K$-groups obtained in this article might be useful for further research. 
\section{Preliminaries and Grayson's definition}

In order to define Grayson's $K$-groups, we need an idea of binary complexes. Let us recall the notion from \cite{Gray} for  exact categories. 
\subsection*{Binary chain complexes}
Let $\mathcal{N}$ denote an exact category. A bounded chain complex $N$ in $\mathcal{N}$ is said to be an {\it acylic chain complex} if each differential $d_{i}: N_{i} \to N_{i-1}$ can be factored as $N_{i} \to Z_{i-1} \to N_{i-1}$ such that each $0\to Z_{i} \to N_{i} \to Z_{i-1}\to 0$ is a short exact sequence of $\mathcal{N}.$ Let $C\mathcal{N}$ denote the category of bounded chain complexes in $\mathcal{N}.$ The full subcategory of $C\mathcal{N}$ consisting of  bounded acylic complexes in $C\mathcal{N}$ is denoted by $C^{q}\mathcal{N}.$ The category $C^{q}\mathcal{N}$ is exact.

A chain complex in $\mathcal{N}$ with two differentials (not necessarily distinct) is called a {\it binary chain complex} in $\mathcal{N}.$ In other words, it is a triple $(N_{*}, d, d^{'})$ with $(N_{*}, d)$ and $(N_{*}, d^{'})$ are in $C\mathcal{N}.$ If $d=d^{'}$ then we call a diagonal binary complex. A morphism between two binary complexes $(N_{*}, d, d^{'})$ and  $(\tilde{N}_{*}, \tilde{d}, \tilde{d}^{'})$ is a morphism between the underlying graded objects $N$ and $\tilde{N}$ that commutes with both differentials. The category of bounded binary complexes in $\mathcal{N}$ is denoted by $B\mathcal{N}.$ There is always a diagonal functor (see Definition 3.1 of \cite{Gray}) $$ \Delta: C\mathcal{N} \to B\mathcal{N}, ~{\rm defined ~by}~ \Delta((N_{*}, d))= (N_{*}, d, d).$$ As before, let $B^{q}\mathcal{N}$ denote the full subcategory of $B\mathcal{N}$ whose objects are bounded acylic binary complexes. This is also an exact category.

By iterating, one can define exact category $(B^{q})^{n}\mathcal{N}=B^{q}B^{q}\cdots B^{q}\mathcal{N}$ for each $n\geq 0.$ An object of the exact category  $(B^{q})^{n}\mathcal{N}$ of bounded acylic binary multicomplexes of dimension $n$ in $\mathcal{N}$ is a bounded $\mathbb{Z}^{n}$- graded collection of objects of $\mathcal{N},$ together with a pair of acyclic differentials $d^{i}$ and $\tilde{d^{i}}$ in each direction $1\leq i\leq n,$ where the differentials $d^{i}$ and $\tilde{d^{i}}$ commute with $d^{j}$ and $\tilde{d^{j}}$ whenever $i\neq j.$ Thus, a typical object looks like $(N_{*}, (d^{1}, \tilde{d^{1}}), (d^{2}, \tilde{d^{2}}), \dots, (d^{n}, \tilde{d^{n}})),$ where $N_{*}$ is a
bounded $\mathbb{Z}^{n}$- graded collection of objects of $\mathcal{N}.$ We say that an acyclic binary multicomplex $(N_{*}, (d^{1}, \tilde{d^{1}}), (d^{2}, \tilde{d^{2}}), \dots, (d^{n}, \tilde{d^{n}}))$ is {\it diagonal} if $d^{i}=\tilde{d^{i}}$ for at least one $i.$ 
\subsection*{Grayson's Definition}
In \cite{Nenasheb}, Nenashev gave a description of $K_{1}$-group in terms of generator and relations using the notion of double exact sequences. Motivated by \cite{Nenasheb},  Grayson defined higher $K$-groups in terms of generators and relations using binary complexes (see \cite{Gray}). However, Nenashev's  $K_{1}$-group agrees with Grayson's $K_{1}$-group (see Corollary 4.2 of \cite{KKW}). In the rest of the paper, we assume the following as the definition of higher $K$-groups.

\begin{definition}\label{main def}(see Corollary 7.4 of \cite{Gray}) Let $\mathcal{N}$ be an exact category. For $n\geq 1,$ $K_{n}\mathcal{N}$ is the abelian group having generators $[N],$ one for each object $N$ of $(B^{q})^{n}\mathcal{N}$ and the relations are:
\begin{enumerate}
 \item $[N^{'}] + [N^{''}]=[N]$ for every short exact sequence $0 \to N^{'} \to N \to N^{''}\to 0$ in $(B^{q})^{n}\mathcal{N};$
 \item $[T]=0$ if $T$ is a diagonal acyclic binary multicomplex.
\end{enumerate}
\end{definition}
Note that if we only consider the relation (1) then it is just $K_{0}(B^{q})^{n}\mathcal{N}.$

\begin{definition}\label{diad sub} We denote $T_{\mathcal{N}}^{n}$ as the subgroup of $K_{0}(B^{q})^{n}\mathcal{N}$ generated by the $K_{0}$-classes of all diagonal acyclic binary multicomplexes in $(B^{q})^{n}\mathcal{N}.$  
 
\end{definition}

\begin{remark}\label{k grp as a qt}{\rm
 In view of  Definition \ref{main def},  we have $K_{n}\mathcal{N}\cong K_{0}(B^{q})^{n}\mathcal{N}/T_{\mathcal{N}}^{n}.$}
\end{remark}

\begin{lemma}\label{harris observation} For each $n\geq 1$, there is a split short exact sequence 
$$0 \to K_{n-1}C^{q}\mathcal{N}\stackrel{\Delta}\to  K_{n-1}B^{q}\mathcal{N} \to K_{n}\mathcal{N}\to 0,$$ which is functorial in $\mathcal{N}.$ 
 \end{lemma}
 \begin{proof}
  See Lemma 2.7 of \cite{Harris}.
 \end{proof}
\begin{remark}\label{further observ to harris}{\rm
 For $n=1,$ $ K_{0}C^{q}\mathcal{N}\cong {\rm im}(\Delta)=T_{\mathcal{N}}^{1}.$ Thus, the above lemma implies that $K_{0}B^{q}\mathcal{N}\cong T_{\mathcal{N}}^{1} \oplus K_{1}\mathcal{N}.$ }
 
 \end{remark}

\subsection*{Higher Nil $K$-groups via binary complexes}
The category ${\bf Nil}(R)$ is exact. By Remark \ref{k grp as a qt}, we can view the $n$-th Nil $K$-group as a quotient of the Grothendieck group of the category $(B^{q})^{n} {\bf Nil}(R).$ More precisely,  $K_{n}{\bf Nil}(R)\cong K_{0}(B^{q})^{n} {\bf Nil}(R)/T_{{\bf Nil}(R)}^{n}$ for $n\geq 0.$ Here, $T_{{\bf Nil}(R)}^{n}$ is a subgroup of $K_{0}(B^{q})^{n} {\bf Nil}(R)$, as described in Definition \ref{diad sub}. Moreover, there is a split exact sequence
 
 $$ 0\to {\rm Nil}_{n}(R) \to K_{0}(B^{q})^{n} {\bf Nil}(R)/T_{{\bf Nil}(R)}^{n} \to K_{0}(B^{q})^{n} {\bf P}(R)/T_{{\bf P}(R)}^{n} \to 0.$$
 Let $\widetilde{{\rm Nil}}_{n}(R)$(resp. $\widetilde{T}_{R}^{n}$) denote the kernel of the map $K_{0}(B^{q})^{n} {\bf Nil}(R) \to K_{0}(B^{q})^{n} {\bf P}(R)$ (resp. $T_{{\bf Nil}(R)}^{n} \to  T_{{\bf P}(R)}^{n}$). Clearly, the map $T_{{\bf Nil}(R)}^{n} \to  T_{{\bf P}(R)}^{n}$ is surjective. We have the following result, which will be used in section \ref{Nil1}.
 \begin{lemma}\label{useful for nil1}
        For $n\geq 1,$ there is a canonical short exact sequence 
        \begin{equation}\label{seq for nil}
         0 \to \widetilde{T}_{R}^{n} \to \widetilde{{\rm Nil}}_{n}(R) \to {\rm Nil}_{n}(R) \to 0.
        \end{equation} If $n=1$ then this sequence splits, i.e., $\widetilde{{\rm Nil}}_{1}(R)\cong {\rm Nil}_{1}(R)\oplus \widetilde{T}_{R}^{1}.$

       \end{lemma}
 \begin{proof}
  The assertion follows by chasing the following commutative diagram
 $$ \begin{CD}
        @.  0 @. 0 @.0  \\
           @. @VVV  @VVV  @VVV  \\
         0 @>>> \widetilde{T}_{R}^{n} @>>> T_{{\bf Nil}(R)}^{n} @>>> T_{{\bf P}(R)}^{n}@>>> 0 \\
            @.  @VVV    @VVV   @VVV \\
        0 @>>> \widetilde{{\rm Nil}}_{n}(R) @>>> K_{0}(B^{q})^{n} {\bf Nil}(R) @>>> K_{0}(B^{q})^{n} {\bf P}(R) @>>> 0 \\
             @. @VVV   @VVV    @VVV  \\
        0 @>>> {\rm Nil}_{n}(R) @>>>K_{n}{\bf Nil}(R) @>>> K_{n}(R) @>>> 0 \\
         @.  @VVV   @VVV  @VVV  \\
        @. 0 @. 0 @. 0,
       \end{CD}$$ where  rows are split exact sequences, the second and third columns are  exact sequences. For $n=1,$ the second column is split exact (see Remark \ref{further observ to harris}). This forces that the first column is split exact provided $n=1.$
 \end{proof}

\section{Vanishing of zeroth Nil $K$-groups}\label{about nil}
 We know  ${\rm Nil}_{0}(R)\cong NK_{1}(R)=\ker [K_{1}(R[t]) \stackrel{t\mapsto 0}\to K_{1}(R)]$ (for instance, see Proposition III.3.5.3 of \cite{wei 1}). The homotopy invariance of $K$-theory is known for regular noetherian rings and  valuation rings (see \cite{KM}). Thus, ${\rm Nil}_{0}(R)=0$ provided $R$ is a regular noetherian or valuation ring. In this section, we discuss a condition on $R$ under which ${\rm Nil}_{0}(R)$ is trivial. 

 For a fix $n,$ $(R^{n}, \nu)$ is an object in ${\bf Nil}(R).$ Assume that $\nu$ is non-zero nilpotent. Since $\nu$ is nilpotent, there exist a least $m\in \mathbb{N}$ such that $\nu^{m}=0$ and $\nu^{r}\neq 0$ for $r< m.$ Then we have a chain of $R$-modules
$$0 \subseteq \ker(\nu)\subseteq \ker(\nu^{2}) \subseteq \dots \subseteq \ker(\nu^{m-1})\subseteq \ker(\nu^{m})=R^{n}.$$

\begin{lemma}\label{tor free}
  $\frac{\ker(\nu^{i+1})}{ker(\nu^{i})} $ is a torsion free $R$-module for $1\leq i \leq m-1.$
\end{lemma}
\begin{proof}
 Let $x + ker(\nu^{i})$ be a torsion element of $\frac{\ker(\nu^{i+1})}{ker(\nu^{i})}.$ Then there exist a non-zero-divisor $r\in R$ such that $rx\in ker(\nu^{i}).$ This implies that $r\nu^{i}(x)=0$ in $R^{n}.$ So, $x\in ker(\nu^{i}).$
\end{proof}

 \begin{lemma}\label{gen becomes zero}
 $[(R^{n}, \nu)]=[(R^{n}, 0)]$ in $ K_{0}({\bf Nil}(R))$ provided every finitely generated torsion free $R$-module is projective.
\end{lemma}
\begin{proof}
 We have an exact sequence of $R$-modules
 $$0\to \ker(\nu^{m-1}) \to R^{n} \to \frac{R^{n}}{\ker(\nu^{m-1})} \to 0.$$ By Lemma \ref{tor free}, $\frac{R^n}{\ker(\nu^{m-1})}$ is a finitely generated projective $R$-modules. Thus the sequence splits, and we get that  $\ker(\nu^{m-1})$ is also a finitely generated projective $R$-module. By considering the exact sequence $ 0\to \ker(\nu^{m-2}) \to \ker(\nu^{m-1}) \to \frac{\ker(\nu^{m-1})}{\ker(\nu^{m-2})}\to 0$ and using Lemma \ref{tor free}, we obtain $\ker(\nu^{m-2})$ is  a finitely generated projective $R$-module. Continuing this way, each $\ker(\nu^{i})$ is  a finitely generated projective $R$-module. Observe that the following diagram of exact sequences 
 $$\begin{CD}
    0 @>>> \ker(\nu^{i}) @>>> \ker(\nu^{i+1}) @>>> \frac{\ker(\nu^{i+1})}{\ker(\nu^{i})}@>>> 0\\
    @. @V \nu VV   @V\nu VV   @V \bar{\nu}=0 VV \\
    0 @>>> \ker(\nu^{i}) @>>> \ker(\nu^{i+1}) @>>> \frac{\ker(\nu^{i+1})}{\ker(\nu^{i})}@>>> 0
   \end{CD}$$
is commutative for each $i,$ $1\leq i \leq m-1.$ Thus, we have an exact sequence 
$$ 0 \to (\ker(\nu^{i}), \nu) \to (\ker(\nu^{i+1}), \nu) \to (\frac{\ker(\nu^{i+1})}{\ker(\nu^{i})}, 0) \to 0$$ in ${\bf Nil}(R)$ for  $1\leq i \leq m-1.$ In $ K_{0}({\bf Nil}(R))$, we get
\begin{align*}
 [(R^{n}, \nu)] & = [(\ker(\nu^{m-1}), \nu)] + [(\frac{R^{n}}{\ker(\nu^{m-1})}, 0)]\\
 & = [(\ker(\nu^{m-2}), \nu)] + [(\frac{\ker(\nu^{m-1})}{\ker(\nu^{m-2})}, 0)] + [(\frac{R^{n}}{\ker(\nu^{m-1})}, 0)]\\
 &=  [(\ker(\nu^{2}), \nu)] + \dots + [(\frac{\ker(\nu^{m-1})}{\ker(\nu^{m-2})}, 0)] + [(\frac{R^{n}}{\ker(\nu^{m-1})}, 0)] \\
 &= [(\ker(\nu)\oplus \frac{\ker(\nu^{2})}{\ker(\nu)} \oplus \dots \oplus \frac{R^{n}}{\ker(\nu^{m-1})}, 0)]\\
 & = [(R^{n}, 0)]. 
 \end{align*}\end{proof}
\begin{theorem}\label{nil zero vanish}
 Let $R$ be a commutative ring with unity. Assume that every finitely generated torsion free $R$-module is projective. Then ${\rm Nil}_{0}(R)=0.$
\end{theorem}

\begin{proof}
 Recall that ${\rm Nil}_{0}(R)$ is generated by elements of the form $[(R^{n}, \nu)]- [(R^{n}, 0)]$ for some $n$ and some nilpotent endomorphism $\nu.$ Lemma \ref{gen becomes zero} yields the result.
\end{proof}

\begin{remark}\label{int val pol}{\rm
 The hypothesis on $R$ in the above theorem holds for any {\it Pr\"{u}fer} domain. We say that a ring $R$ is a {\it Pr\"{u}fer} domain if $R_{\mathfrak{p}}$ is a valuation domain for all prime ideals $\mathfrak{p}$ of $R.$ Clearly, a valuation ring is {\it Pr\"{u}fer}. A domain $R$ is {\it Pr\"{u}fer} if and only if every finitely generated torsion free $R$-module is projective. The ring of integer-valued polynomials ${\rm Int}(\mathbb{Z})= \{f\in \mathbb{Q}[x]| f(\mathbb{Z})\subset \mathbb{Z}\}$ is a Pr\"{u}fer domain. In fact, it is a non-noetherian Pr\"{u}fer domain (see \cite{cahen}). One can see \cite{SS} for $K$-theory of Pr\"{u}fer domains.
 }
\end{remark}

\section{Cofinality Lemma}\label{key lemmas}
Let $\mathcal{N}$ be an exact category. An exact subcategory $\mathcal{M}$ of $\mathcal{N}$ is called {\it closed under extensions} whenever there exists a short exact sequence $0 \to N_{1} \to N \to N_{2}\to 0$ in $\mathcal{N}$ with $N_{1}$ and $N_{2}$ in $\mathcal{M}$ then $N$ is isomorphic to an object of $\mathcal{M}.$ We say that an exact subcategory $\mathcal{M}$ of $\mathcal{N}$ is {\it cofinal} in $\mathcal{N}$ if for every object $N_{1}\in \mathcal{N}$ there exists $N_{2}\in \mathcal{N}$ such that $N_{1}\oplus N_{2}$ is isomorphic to an object of $\mathcal{M}.$ Let ${\bf Free}(R)$ denote the category of finitely generated free $R$-modules. Clearly, ${\bf Free}(R)$ is an exact subcategory of ${\bf P}(R)$ which is cofinal and closed under extensions.

We can define a category ${\bf Nil(Free}(R))$ whose object are pairs $(F, \nu),$ where $F$ is in ${\bf Free}(R)$ and $\nu$ is a nilpotent endomorphism. A morphism $f: (F_{1}, \nu_{1}) \to (F_{2}, \nu_{2})$ is a $R$-module map $f: F_{1} \to F_{2}$ such that $f\nu_{1}=\nu_{2}f.$ One can check that ${\bf Nil(Free}(R))$ is an exact category.
\begin{lemma}\label{nilfree cofinal}
 The category ${\bf Nil(Free}(R))$ is an exact full subcategory of ${\bf Nil}(R)$ which is cofinal and closed under extensions. 
\end{lemma}

\begin{proof}
 Let $(P, \nu)\in {\bf Nil}(R).$ Then there exists a $Q$ in ${\bf P}(R)$ such that $\alpha: P\oplus Q\cong R^{n}$ for some $n>0.$ Note that $(Q, 0)\in {\bf Nil}(R).$ So, we get $ (P, \nu)\oplus (Q, 0)\cong (R^{n}, \nu^{'}),$ where $\nu^{'}= \alpha(\nu , 0)\alpha^{-1}.$ This implies that ${\bf Nil(Free}(R))\subseteq {\bf Nil}(R)$ is cofinal.
 
 Suppose that the sequence 
 $$ 0 \to (P_{1}, \nu_{1}) \to (P, \nu) \to (P_{2}, \nu_{2}) \to 0$$ is exact with $(P_{1}, \nu_{1}), (P_{2}, \nu_{2}) \in {\bf Nil(Free}(R)),$ i.e., the following diagram
 $$\begin{CD}
  0 @>>> P_{1} @>>> P @>>> P_{2} @>>> 0\\
  @. @V\nu_{1} VV  @V \nu VV @V\nu_{2} VV \\
  0 @>>> P_{1} @>>> P @>>> P_{2} @>>> 0
 \end{CD}$$
is commutative with exact rows and $P_{1}, P_{2} \in {\bf Free}(R).$ Let $\beta$ denote the isomorphism $P\cong P_{1}\oplus P_{2}.$ We define a nilpotent endomorphism of $P_{1}\oplus P_{2}$ as $\beta \nu \beta^{-1}.$ Therefore, $(P, \nu)\cong (P_{1}\oplus P_{2}, \beta \nu \beta^{-1})$ in ${\bf Nil(Free}(R)).$ Hence the lemma. 
\end{proof}
We now consider categories $B^{q}{\bf Nil(Free}(R))$ and $B^{q}{\bf Nil}(R).$ Note that one can identify $B^{q}{\bf Nil(Free}(R))$ with ${\bf Nil}(B^{q}{\bf Free}(R))$ and $B^{q}{\bf Nil}(R)$ with ${\bf Nil}(B^{q}{\bf P}(R))$ because nilpotent endomorphism commutes with each differential.

\begin{lemma}\label{gray cor}
 
 The categories $C^{q}{\bf Nil(Free}(R))\subseteq C^{q}{\bf Nil}(R)$ and $B^{q}{\bf Nil(Free}(R))\subseteq B^{q}{\bf Nil}(R)$ both are cofinal and closed under extensions.
\end{lemma}
\begin{proof}
 By Lemma 6.2 of \cite{Gray}, we know that if $\mathcal{M}$ is an exact full subcategory of $\mathcal{N}$ which is cofinal and  closed under extensions then for every $(N_{*}, d, d^{'})\in B^{q}\mathcal{N}$ there is an object $(L_{*}, e)\in C^{q}\mathcal{N}$ such that $(N_{*}, d, d^{'})\oplus (L_{*}, e, e)\cong(M_{*}, f, f^{'})\in B^{q}\mathcal{M}.$
 The assertion now  follows from Lemma \ref{nilfree cofinal}.
\end{proof}

\section{Generators of ${\rm Nil_{n>0}}(R)$}\label{Nil1}
We first describe generators for ${\rm Nil_{1}}(R)(\cong NK_{2}(R))$. Before that we need some preparations. Let us begin with the following Lemma.
\begin{lemma}\label{insert zero}
 If $[(P_{*}, p, p^{'})]=[(Q_{*}, q, q^{'})]$ in $K_{0}B^{q}{\bf P}(R)$ then $[(P_{*}, p, p^{'}, 0)]=[(Q_{*}, q, q^{'}, 0)]$ in $K_{0}B^{q}{\bf Nil}(R).$
\end{lemma}
\begin{proof}
 Since $[(P_{*}, p, p^{'})]=[(Q_{*}, q, q^{'})]$ in $K_{0}B^{q}{\bf P}(R),$ there are short exact sequences (see Exercise II.7.2 of \cite{wei 1}) $$ 0 \to (C_{*}, c, \tilde{c}) \to (A_{*}, a, \tilde{a}) \to (D_{*}, d, \tilde{d}) \to 0$$ and  $$0 \to (C_{*}, c, \tilde{c}) \to (B_{*}, b, \tilde{b}) \to (D_{*}, d, \tilde{d}) \to 0$$ in $B^{q}{\bf P}(R)$ such that 
 \begin{equation}\label{iso}
        (P_{*}, p, p^{'})\oplus (A_{*}, a, \tilde{a})\cong (Q_{*}, q, q^{'})\oplus (B_{*}, b, \tilde{b}).
        \end{equation}
Note that  $$ 0 \to (C_{*}, c, \tilde{c}, 0) \to (A_{*}, a, \tilde{a}, 0) \to (D_{*}, d, \tilde{d}, 0) \to 0$$ and  $$0 \to (C_{*}, c, \tilde{c}, 0) \to (B_{*}, b, \tilde{b}, 0) \to (D_{*}, d, \tilde{d}, 0) \to 0$$ both are short exact sequences in $B^{q}{\bf Nil}(R).$ Thus, $[(A_{*}, a, \tilde{a}, 0)]=[(B_{*}, b, \tilde{b}, 0)]= [(C_{*}, c, \tilde{c}, 0)]+ [(D_{*}, d, \tilde{d}, 0)]$ in $K_{0}B^{q}{\bf Nil}(R).$ By using the isomorphism (\ref{iso}), we get $ (P_{*}, p, p^{'}, 0)\oplus (A_{*}, a, \tilde{a}, 0)\cong (Q_{*}, q, q^{'}, 0)\oplus (B_{*}, b, \tilde{b}, 0).$ Hence the assertion. 
                                                                                                                                                                                                                                \end{proof}

                                                                                                                                                                                                                                        \begin{lemma}\label{nil can be choosen zero}
 For every $(P_{*}, p, p^{'}, \nu)$ in $B^{q}{\bf Nil}(R)$ there exists a $(Q_{*}, q, q, 0)$ in $B^{q}{\bf Nil}(R)$ such that $(P_{*}, p, p^{'}, \nu)\oplus (Q_{*}, q, q, 0)\cong (F_{*}, f, f^{'}, \nu^{'}),$ where $(F_{*}, f, f^{'}, \nu^{'})\in B^{q}{\bf Nil(Free}(R)).$                                                                                                                                                                                                                                        
                                                                                                                                                                                                                                    \end{lemma}
                                                                                                                                                                                                                                    \begin{proof}
 Note that $(P_{*}, p, p^{'})\in B^{q}{\bf P}(R).$ By Lemma 6.2 of \cite{Gray}, there exists a $(Q_{*}, q, q)$ such that $(P_{*}, p, p^{'})\oplus (Q_{*}, q, q)\cong (F_{*}, f, f^{'}),$ where $(F_{*}, f, f^{'})\in B^{q}{\bf Free}(R).$ Let $\alpha$ denote the above isomorphism. Define $\nu^{'}:= \alpha(\nu, 0) \alpha^{-1}.$ The isomorphisms $\alpha$ and $\alpha^{-1}$ commute with differentials. We need to check that $\nu^{'}$ is a nilpotent endomorphism on $(F_{*}, f, f^{'})$. Let us just consider one case. Checking for other case is similar. We have $\alpha(p, q)= f\alpha$ and $(p, q)\alpha^{-1}= \alpha^{-1}f.$ Moreover, $\nu$ commutes with $p$ and $p^{'}.$ Then $f\nu^{'}= f\alpha(\nu, 0) \alpha^{-1} = \alpha(p, q)(\nu, 0) \alpha^{-1}= \alpha(\nu, 0)(p, q) \alpha^{-1}=\alpha(\nu, 0)\alpha^{-1}f=\nu^{'}f.$ This shows that $\nu^{'}$ is a nilpotent endomorphism on $(F_{*}, f, f^{'})$ and we get the desired isomorphism $(P_{*}, p, p^{'}, \nu)\oplus (Q_{*}, q, q, 0)\cong (F_{*}, f, f^{'}, \nu^{'}).$ 
                                                                                                                                                                                                                                    \end{proof}

Recall that we have the forgetful map $$K_{0}B^{q} {\bf Nil}(R) \to K_{0}B^{q}{\bf P}(R), [(P_{*}, p, p^{'}, \nu)]\mapsto [(P_{*}, p, p^{'})]$$ and $\widetilde{{\rm Nil}}_{1}(R)$ denotes the kernel of the above map. 

\begin{lemma}\label{gen of tilde nil}
 The group $\widetilde{{\rm Nil}}_{1}(R)$ is generated by elements of the form $[(F_{*}, f, f^{'}, \nu)]- [(F_{*}, f, f^{'}, 0)],$ where $(F_{*}, f, f^{'}, \nu)$ and $(F_{*}, f, f^{'}, 0)$ both are in $B^{q}{\bf Nil(Free}(R)).$
\end{lemma}
\begin{proof}
 By Lemma \ref{gray cor}, $B^{q}{\bf Nil(Free}(R))\subseteq B^{q}{\bf Nil}(R)$ is cofinal with closed under extensions. Let $x\in K_{0}(B^{q}) {\bf Nil}(R).$ Then $x$ is of the form $[(P_{*}, p, p^{'}, \nu)] - [(F_{*}, f, f^{'}, \tilde{\nu})],$ where $(P_{*}, p, p^{'}, \nu)\in B^{q}{\bf Nil}(R)$ and $(F_{*}, f, f^{'}, \tilde{\nu})\in B^{q}{\bf Nil(Free}(R))$ (see Remark II.7.2.1 of \cite{wei 1}). If $x \in \widetilde{{\rm Nil}}_{1}(R)$ then \begin{equation}\label{x in Nil}
        [(P_{*}, p, p^{'})]= [(F_{*}, f, f^{'})].
        \end{equation} Since $(P_{*}, p, p^{'}) \in B^{q}{\bf P}(R),$ there exists a $(Q_{*}, q, q)\in B^{q}{\bf P}(R)$ such that $[(P_{*}, p, p^{'})]+ [(Q_{*}, q, q)]= [(\tilde{F}_{*}, \tilde{f}, \tilde{f}^{'})]$ in $K_{0}B^{q}{\bf P}(R)$ (by Lemma 6.2 of \cite{Gray}). Here $(\tilde{F}_{*}, \tilde{f}, \tilde{f}^{'}) \in B^{q}{\bf Free}(R).$ We also have (see Lemma \ref{nil can be choosen zero})
        \begin{equation}\label{zero can be put}
        [(P_{*}, p, p^{'}, \nu)]+ [(Q_{*}, q, q, 0)]= [(\tilde{F}_{*}, \tilde{f}, \tilde{f}^{'}, \nu^{'})].
        \end{equation} By (\ref{x in Nil}), $[(F_{*}, f, f^{'})]+ [(Q_{*}, q, q)]= [(\tilde{F}_{*}, \tilde{f}, \tilde{f}^{'})]$ in $K_{0}B^{q}{\bf P}(R).$ Using Lemma \ref{insert zero}, we get  $[(F_{*}, f, f^{'}, 0)]+ [(Q_{*}, q, q, 0)]= [(\tilde{F}_{*}, \tilde{f}, \tilde{f}^{'}, 0)]$ in $K_{0}B^{q}{\bf Nil}(R).$ Now, (\ref{zero can be put}) implies that \begin{equation}\label{ before final expression}
        [(P_{*}, p, p^{'}, \nu)]= [(\tilde{F}_{*}, \tilde{f}, \tilde{f}^{'}, \nu^{'})]- [(\tilde{F}_{*}, \tilde{f}, \tilde{f}^{'}, 0)]+ [(F_{*}, f, f^{'}, 0)].
        \end{equation} Therefore,\begin{equation}\label{ final expression}
        x= ([(\tilde{F}_{*}, \tilde{f}, \tilde{f}^{'}, \nu^{'})]- [(\tilde{F}_{*}, \tilde{f}, \tilde{f}^{'}, 0)])- ([(F_{*}, f, f^{'}, \tilde{\nu})]- [(F_{*}, f, f^{'}, 0)]).
        \end{equation} This shows that $\widetilde{{\rm Nil}}_{1}(R)$ is generated by elements of the form $[(F_{*}, f, f^{'}, \nu)]- [(F_{*}, f, f^{'}, 0)],$ where $(F_{*}, f, f^{'}, \nu)$ and $(F_{*}, f, f^{'}, 0)$ both are in $B^{q}{\bf Nil(Free}(R)).$
\end{proof}

        \begin{theorem}\label{gen for Nil1}
         The group ${\rm Nil}_{1}(R)$ is generated by elements of the form $$[(F_{*}, f, f^{'}, \nu)]- [(F_{*}, f, f^{'}, 0)],$$ where $(F_{*}, f, f^{'}, \nu), (F_{*}, f, f^{'}, 0) \in B^{q}{\bf Nil(Free}(R))$ and $f\neq f^{'}.$
        \end{theorem}
        
        \begin{proof}
         We know $\widetilde{{\rm Nil}}_{1}(R)\cong {\rm Nil}_{1}(R)\oplus \widetilde{T}_{R}^{1}$ (see Lemma \ref{useful for nil1}). Hence the assertion by Lemma \ref{gen of tilde nil}.
        \end{proof}

        \vspace{1cm}
        \subsection*{Generators of ${\rm Nil}_{n}(R)$ for $n>1$}
        We consider the case ${\rm Nil}_{n}(R) (\cong NK_{n+1}(R))$ for $n>1.$
        \begin{lemma}\label{insert zero for n-complex}
 If $[(P_{*}, (p_{1}, p_{1}^{'}), (p_{2}, p_{2}^{'}), \dots, (p_{n}, p_{n}^{'})) ]=[(Q_{*}, (q_{1}, q_{1}^{'}), (q_{2}, q_{2}^{'}), \dots, (q_{n}, q_{n}^{'}))]$ in $K_{0}(B^{q})^{n}{\bf P}(R)$ then $$[(P_{*}, (p_{1}, p_{1}^{'}), (p_{2}, p_{2}^{'}), \dots, (p_{n}, p_{n}^{'}),0)]=[(Q_{*}, (q_{1}, q_{1}^{'}), (q_{2}, q_{2}^{'}), \dots, (q_{n}, q_{n}^{'}), 0)]$$ in $K_{0}(B^{q})^{n}{\bf Nil}(R).$
\end{lemma}
\begin{proof}
 The proof is similar to Lemma \ref{insert zero}. More precisely, just rewrite the proof of Lemma \ref{insert zero} for $(B^{q})^{n}{\bf P}(R)$.
\end{proof}

\begin{lemma}\label{nil can be choosen zero for n-complexes}
 For every $(P_{*}, (p_{1}, p_{1}^{'}), (p_{2}, p_{2}^{'}), \dots, (p_{n}, p_{n}^{'}), \nu)$ in $(B^{q})^{n}{\bf Nil}(R)$ there exists a $(Q_{*}, (q, q), (q_{2}, q_{2}^{'}), \dots, (q_{n}, q_{n}^{'}), 0)$ in $(B^{q})^{n}{\bf Nil}(R)$ such that \tiny $$(P_{*}, (p_{1}, p_{1}^{'}), (p_{2}, p_{2}^{'}), \dots, (p_{n}, p_{n}^{'}), \nu)\oplus (Q_{*}, (q, q), (q_{2}, q_{2}^{'}), \dots, (q_{n}, q_{n}^{'}), 0)\cong (F_{*}, (f_{1}, f_{1}^{'}), (f_{2}, f_{2}^{'}), \dots, (f_{n}, f_{n}^{'}),  \nu^{'}),$$ \normalsize where $(F_{*}, (f_{1}, f_{1}^{'}), (f_{2}, f_{2}^{'}), \dots, (f_{n}, f_{n}^{'}), \nu^{'})$ is in $(B^{q})^{n}{\bf Nil(Free}(R)).$                                                                                                                                                                                                                                        
                                                                                                                                                                                                                                    \end{lemma}
                                                                                                                                                                                                                                    
                                                                                                                                                                                                                                    \begin{proof}
  By repeatadly using Lemma \ref{gray cor}, $(B^{q})^{n-1}{\bf Free}(R)\subseteq (B^{q})^{n-1}{\bf P}(R)$ is cofinal and closed under extensions. By Lemma 6.2 of \cite{Gray}, for  $(P_{*}, (p_{1}, p_{1}^{'}), (p_{2}, p_{2}^{'}), \dots, (p_{n}, p_{n}^{'}))\in (B^{q})^{n}{\bf P}(R)$ there exists a $(Q_{*}, (q, q), (q_{2}, q_{2}^{'}), \dots, (q_{n}, q_{n}^{'}))$ such that                                                                                                                                                                                                                                  
             \tiny $$(P_{*}, (p_{1}, p_{1}^{'}), (p_{2}, p_{2}^{'}), \dots, (p_{n}, p_{n}^{'}))\oplus (Q_{*}, (q, q), (q_{2}, q_{2}^{'}), \dots, (q_{n}, q_{n}^{'}))\stackrel{\alpha}\cong (F_{*}, (f_{1}, f_{1}^{'}), (f_{2}, f_{2}^{'}), \dots, (f_{n}, f_{n}^{'})),$$ \normalsize where $(F_{*}, (f_{1}, f_{1}^{'}), (f_{2}, f_{2}^{'}), \dots, (f_{n}, f_{n}^{'}))\in (B^{q})^{n}{\bf Free}(R).$ Define $\nu^{'}:= \alpha(\nu, 0) \alpha^{-1}.$ Note that $\alpha$ and $\alpha^{-1}$ commute with differentials in each direction. The rest of the argument is similar to Lemma \ref{nil can be choosen zero}.                                                                                                                                                                                                                    \end{proof}
             Recall that $\widetilde{{\rm Nil}}_{n}(R)$ denotes the kernel of $$K_{0}(B^{q})^{n} {\bf Nil}(R) \to K_{0}(B^{q})^{n}{\bf P}(R),$$  $$[(P_{*}, (p_{1}, p_{1}^{'}), (p_{2}, p_{2}^{'}), \dots, (p_{n}, p_{n}^{'}), \nu)]\mapsto [(P_{*}, (p_{1}, p_{1}^{'}), (p_{2}, p_{2}^{'}), \dots, (p_{n}, p_{n}^{'})].$$ 
             
             \begin{lemma}
              The group  $\widetilde{{\rm Nil}}_{n}(R)$ is generated by elements of the form
$$[(F_{*}, (f_{1}, f_{1}^{'}), (f_{2}, f_{2}^{'}), \dots, (f_{n}, f_{n}^{'}), \nu)]- [(F_{*}, (f_{1}, f_{1}^{'}), (f_{2}, f_{2}^{'}), \dots, (f_{n}, f_{n}^{'}), 0)],$$ where $(F_{*}, (f_{1}, f_{1}^{'}), (f_{2}, f_{2}^{'}), \dots, (f_{n}, f_{n}^{'}), \nu)$ and $(F_{*}, (f_{1}, f_{1}^{'}), (f_{2}, f_{2}^{'}), \dots, (f_{n}, f_{n}^{'}), 0)$ both are in $(B^{q})^{n}{\bf Nil(Free}(R)).$
             \end{lemma}

      \begin{proof}       By Lemma \ref{gray cor}, $(B^{q})^{n}{\bf Nil(Free}(R))\subseteq (B^{q})^{n}{\bf Nil}(R)$ is cofinal with closed under extensions. Let $x\in K_{0}(B^{q})^{n} {\bf Nil}(R).$ By applying Remark II.7.2.1 of \cite{wei 1} for the categories $(B^{q})^{n}{\bf Nil(Free}(R))\subseteq (B^{q})^{n}{\bf Nil}(R),$ we get $$x=  [(P_{*}, (p_{1}, p_{1}^{'}), (p_{2}, p_{2}^{'}), \dots, (p_{n}, p_{n}^{'}), \nu)]- [(F_{*}, (f_{1}, f_{1}^{'}), (f_{2}, f_{2}^{'}), \dots, (f_{n}, f_{n}^{'}), \tilde{\nu})],~ {\rm where}~ $$  $(P_{*}, (p_{1}, p_{1}^{'}), (p_{2}, p_{2}^{'}), \dots, (p_{n}, p_{n}^{'}), \nu) \in (B^{q})^{n}{\bf Nil}(R)$ and $(F_{*}, (f_{1}, f_{1}^{'}), (f_{2}, f_{2}^{'}), \dots, (f_{n}, f_{n}^{'}), \tilde{\nu})\\ \in (B^{q})^{n}{\bf Nil(Free}(R)).$  Suppose  $x\in \widetilde{{\rm Nil}}_{n}(R)$. Thus, $$[(P_{*}, (p_{1}, p_{1}^{'}), (p_{2}, p_{2}^{'}), \dots, (p_{n}, p_{n}^{'}))]= [(F_{*}, (f_{1}, f_{1}^{'}), (f_{2}, f_{2}^{'}), \dots, (f_{n}, f_{n}^{'}))]$$ in $K_{0}(B^{q})^{n}{\bf P}(R).$ The rest of the argument is similar to the case of $\widetilde{{\rm Nil}}_{1}(R)$. By using Lemmas \ref{insert zero for n-complex} and \ref{nil can be choosen zero for n-complexes}, we get \begin{equation}\label{  n complexes expression}\begin{split}
        x= ([(\tilde{F}_{*}, (\tilde{f_{1}}, \tilde{f_{1}}^{'}), (\tilde{f_{2}}, \tilde{f_{2}}^{'}), \dots, (\tilde{f_{n}}, \tilde{f_{n}}^{'}),\nu^{'})]- [(\tilde{F}_{*}, (\tilde{f_{1}}, \tilde{f_{1}}^{'}), (\tilde{f_{2}}, \tilde{f_{2}}^{'}), \dots, (\tilde{f_{n}}, \tilde{f_{n}}^{'}), 0)])\\- ([(F_{*}, (f_{1}, f_{1}^{'}), (f_{2}, f_{2}^{'}), \dots, (f_{n}, f_{n}^{'}), \tilde{\nu})]- [(F_{*}, (f_{1}, f_{1}^{'}), (f_{2}, f_{2}^{'}), \dots, (f_{n}, f_{n}^{'}), 0)]),\end{split}
        \end{equation} where all the entries are in $(B^{q})^{n}{\bf Nil(Free}(R)).$ Therefore, $\widetilde{{\rm Nil}}_{n}(R)$ is generated by elements of the form
$$[(F_{*}, (f_{1}, f_{1}^{'}), (f_{2}, f_{2}^{'}), \dots, (f_{n}, f_{n}^{'}), \nu)]- [(F_{*}, (f_{1}, f_{1}^{'}), (f_{2}, f_{2}^{'}), \dots, (f_{n}, f_{n}^{'}), 0)].$$\end{proof}
By Lemma \ref{useful for nil1}, ${\rm Nil}_{n}(R)\cong \widetilde{{\rm Nil}}_{n}(R)/ \widetilde{T}_{R}^{n}.$ Hence we get, 

\begin{theorem}\label{gen for Nil_n}
 The group ${\rm Nil}_{n}(R)$ is generated by elements of the form \small$$([(F_{*}, (f_{1}, f_{1}^{'}), (f_{2}, f_{2}^{'}), \dots, (f_{n}, f_{n}^{'}), \nu)]- [(F_{*}, (f_{1}, f_{1}^{'}), (f_{2}, f_{2}^{'}), \dots, (f_{n}, f_{n}^{'}), 0)])(~{\rm mod} ~\widetilde{T}_{R}^{n}),$$ \normalsize where $(F_{*}, (f_{1}, f_{1}^{'}), (f_{2}, f_{2}^{'}), \dots, (f_{n}, f_{n}^{'}), \nu)), (F_{*}, (f_{1}, f_{1}^{'}), (f_{2}, f_{2}^{'}), \dots, (f_{n}, f_{n}^{'}), 0))$ are objects of $(B^{q})^{n}{\bf Nil(Free}(R)).$
\end{theorem}
\begin{remark}\label{length 2 enough}{\rm 
 D. Grayson in \cite[Remark 8.1]{Gray} remarked that acyclic binary multicomplexes supported on $[0, 2]^{n}$ suffice to generate whole group $K_{n}\mathcal{N}.$ Recently, D. Kasprowski and C. Winges establish Grayson's remark in \cite{KW}(more precisely, see Theorem 1.3 of \cite{KW}). In view of \cite{KW}, generators of ${\rm Nil}_{n}(R)$ for $n>0$ obtained in Theorems \ref{gen for Nil1} and \ref{gen for Nil_n} can be restricted to  acyclic binary multicomplexes supported on $[0, 2]^{n}.$}
\end{remark}


\begin{thebibliography}{AAA}


 \bibitem{cahen} Paul-Jean Cahen, Jean-Luc Chabert,  {\it What you should know about integer-valued polynomials}, Amer. Math. Monthly {\bf 123} (2016), no. 4, 311-337.
 \bibitem{Gray} D. R. Grayson, {\it Algebraic $K$-theory via binary complexes}, Journal of American Mathematical Society, Vol {\bf 25}, Number {\bf 4} (2012), 1149-1167.
 
 \bibitem{Harris} T. Harris, {\it Algebraic proofs of some fundamental theorems in Algebraic $K$-theory}, Homology, Homotopy and Applications, vol. {\bf 17(1)} (2015), 267-280.
 
 
 
 \bibitem{KKW} D. Kasprowski, B. Kock and C. Winges, {\it $K_{1}$-groups via binary complexes of fixed length}, Homology Homotopy Appl. {\bf 22} (2020), no. 1, 203-213.
 \bibitem{KM} S. Kelly and M. Morrow, {\it K-theory of valuation rings},  Compos. Math. {\bf 157} (2021), no. 6, 1121-1142.
 
 \bibitem{KW} D. Kasprowski and C. Winges, {\it Shortening binary complexes and commutativity of $K$-theory with infinite products}, Transactions of the American Mathematical Society,, Series B
{\bf 7}(2020), 1-23.

 \bibitem{Nenasheb} A. Nenashev,  $K_{1}$ by generators and relations, J.of Pure and Applied Algebra. {\bf 131(2)} (1998), 195-212.
 
 
 
 
\bibitem{SS} S. Banerjee and V. Sadhu, {$K$-theory of Pr\"{u}fer domains},  Arch. Math. (Basel) {\bf 118} (2022), no. 5, 465-470.




\bibitem{wei 1} C. Weibel, The $K$-Book: An Introduction to Algebraic $K$-theory.   Graduate Studies in Mathematics, {\bf 145}, American Mathematical Society, Providence, RI, 2013.

\end{thebibliography}
\end{document}